\documentclass[11pt]{article}
\usepackage{amsmath,amsthm,amssymb,amsfonts,amscd}
\usepackage{amsxtra}
\usepackage{eucal}
\usepackage{mathrsfs}
\usepackage{color}
\usepackage{xcolor}
\usepackage{hyperref}
\usepackage{enumerate}
\usepackage{blkarray}
\usepackage{tikz}
\usetikzlibrary{positioning}
\usepackage{enumitem}
\usepackage{tcolorbox}
\usepackage{algorithm}
\usepackage{algpseudocode}
\usepackage[a4paper,margin=2cm]{geometry}
\definecolor{forestgreen(traditional)}{rgb}{0.0, 0.27, 0.13}
\definecolor{forestgreen(web)}{rgb}{0.13, 0.55, 0.13}
\definecolor{airforceblue}{rgb}{0.36, 0.54, 0.66}

\newcommand{\symdiff}{\triangle}

\newcommand{\cB}{\mathcal{B}}

\newcommand{\bZ}{\mathbb{Z}}

\newcommand{\lift}{\mathrm{lift}}

\hypersetup{
  colorlinks,
  citecolor=forestgreen(web),
  linkcolor=airforceblue,
  urlcolor=magenta,
  pdftitle={Note on Hamiltonicity of basis graphs of even delta-matroids},
  pdfauthor={Donggyu Kim and Sang-il Oum}}
\newenvironment{subproof}[1][\proofname]{%
  \begin{proof}[#1]%
}{%
  \end{proof}%
}

\newtheorem{theorem}{Theorem}[section]
\newtheorem{lemma}[theorem]{Lemma}

\newtheorem{claim}{Claim}
\newtheorem{corollary}[theorem]{Corollary}
\newtheorem{proposition}[theorem]{Proposition}

\theoremstyle{definition}

\theoremstyle{remark}

\usepackage{authblk}
\title{Note on Hamiltonicity of basis graphs of even delta-matroids}
\author[1,2]{Donggyu Kim\thanks{Supported by the Institute for Basic Science (IBS-R029-C1).}}
\author[$*$2]{Sang-il Oum}
\affil[1]{Department of Mathematical Sciences, KAIST, Daejeon, South~Korea}
\affil[2]{Discrete Mathematics Group,
Institute for Basic Science (IBS),
Daejeon,~South~Korea}
\affil[ ]{Email: \texttt{donggyu@kaist.ac.kr}, \texttt{sangil@ibs.re.kr}}
\date{\today}	

\begin{document}
\maketitle
\begin{abstract}
    We show that the basis graph of an even delta-matroid is Hamiltonian if it has more than two vertices.
    More strongly, we prove that for two distinct edges $e$ and $f$ sharing a common end, it has a Hamiltonian cycle using $e$ and avoiding $f$ unless it has at most two vertices or it is a cycle of length at most four.
    We also prove that if the basis graph is not a hypercube graph, then each vertex belongs to cycles of every length $\ell\ge 3$, and each edge belongs to cycles of every length $\ell \ge 4$.
    For the last theorem, we provide two proofs, one of which uses the result of Naddef (1984) on polytopes and the result of Chepoi (2007) on basis graphs of even delta-matroids, and the other is a direct proof using various properties of even delta-matroids.
    Our theorems generalize the analogous results for matroids by Holzmann and Harary (1972) and Bondy and Ingleton (1976).
\end{abstract}

\section{Introduction}

For two sets $X$ and $Y$, we write $X\symdiff Y := (X\setminus Y) \cup (Y\setminus X)$.
A \emph{delta-matroid} is a pair $(E,\cB)$ of a finite set $E$ and a nonempty set $\cB$ of subsets of $E$ satisfying the \emph{symmetric exchange axiom}:
\begin{enumerate}[label=\rm(SEA)]
    \item\label{sea} if $B,B' \in \cB$ and $e\in B\symdiff B'$, then there is $f \in B\symdiff B'$ such that $B \symdiff \{e,f\} \in \cB$.
\end{enumerate}
We call each member of $\cB$ a \emph{base}.
An \emph{even} delta-matroid is a delta-matroid whose bases have cardinalities of the same parity.

Delta-matroids were introduced independently by Bouchet~\cite{Bouchet1987}, Dress and Havel~\cite{DH1986}, and Chandrasekaran and Kabadi~\cite{CK1988} under various names.
They capture several combinatorial properties of both matroids and Eulerian circuits of $4$-regular graphs.
Here we review examples of even delta-matroids:
\begin{description}
    \item[Matroids.]
    A matroid is precisely a delta-matroid whose bases have the same cardinality.
    \item[Representable even delta-matroids.] 
    For an $E\times E$ skew-symmetric matrix $A$ and a subset $X \subseteq E$, all subsets $Y\subseteq E$ with $\det(A[X\symdiff Y]) \ne 0$ form the bases of an even delta-matroid on $E$~\cite{Bouchet1988}.
    Here $A[X]$ means the $X\times X$ submatrix of $A$.
    \item[Eulerian even delta-matroids.]
    Let $G$ be a connected directed graph such that each vertex has in-degree and out-degree $2$.
    For each vertex $v$, we denote two incoming edges by $i^1_v,i^2_v$ and two outgoing edges by $o^1_v,o^2_v$.
    We say a vertex set $X\subseteq V(G)$ is feasible if there is a directed Eulerian circuit $C$ such that for each $v\in X$, $C$ passes $v$ through edges $i^1_v$ and $o^1_v$, and for each $v\in V(G)\setminus X$, $C$ passes $v$ through edges $i^1_v$ and $o^2_v$.
    Then the feasible sets are the bases of an even delta-matroid on $V(G)$~\cite{Bouchet1987}.
    \item[Matching delta-matroids.] For a graph $G$, all vertex subsets inducing subgraphs with perfect matchings form the bases of an even delta-matroid on~$V(G)$~\cite{Bouchet1989b}.
    \item[Graphic delta-matroids]
    A graft is a pair of graph and its vertex subset.
    Let $(G,T)$ be a graft such that every component has a vertex in $T$.
    We say an edge set $F$ of $G$ is feasible if each component of the subgraph $(V(G),F)$ has odd number of vertices in $T$.
    Then the feasible sets are the bases of an even delta-matroid on $E(G)$~\cite{Oum2009}.
\end{description}
We refer to the survey by Moffat~\cite{Moffatt2019} for more examples of delta-matroids.

The \emph{basis graph}, denoted by $BG(M)$, of an even delta-matroid $M$ is a graph on the set of bases of $M$ such that two vertices $B$ and $B'$ are adjacent if and only if $|B\symdiff B'| = 2$.
Basis graphs were first introduced by Cummins~\cite{Cummins1966} for graphic matroids under the name ``tree graphs.''
Basis graphs of matroids were extensively studied in 1970s~\cite{BI1976,HH1972,HNT1973,Maurer1973,Maurer1973b}.
Holzmann, Norton, and Tobey~\cite{HNT1973} showed that basis graphs determine matroids up to isomorphism and component-wise dual.
Maurer~\cite{Maurer1973} provided the conditions deciding whether a given graph is the basis graph of a matroid.
He also showed that for any basis graph $G$ of a matroid, the $2$-complex obtained by gluing disks along all $3$- and $4$-cycles of $G$ is simply connected.
These results for basis graphs of matroids have been extended for those of even delta-matroids~\cite{CCO2015,Chepoi2007,Wenzel1996}.

In 1966, Cummins~\cite{Cummins1966} showed the following theorem for the basis graphs of graphic matroids.
\begin{theorem}[Cummins~\cite{Cummins1966}]\label{thm:Cummins}
    Let $H$ be the basis graph of a graphic matroid.
    Then every edge of $H$ is contained in a Hamiltonian cycle of $H$.
\end{theorem}
In 1972, Holzmann and Harary~\cite{HH1972} extended Theorem~\ref{thm:Cummins} for general matroids as follows.

\begin{theorem}[Holzmann and Harary~\cite{HH1972}]\label{thm:HH}
    Let $M$ be a matroid and let $e$ be an edge of $BG(M)$.
    Then $BG(M)$ has a Hamiltonian cycle using $e$ if $BG(M)$ contains a cycle, and $BG(M)$ has a Hamiltonian cycle avoiding $e$ if $BG(M)$ contains at least two cycles.
\end{theorem}

Here is our first theorem.

\begin{theorem}\label{thm:Ham}
    Let $M$ be an even delta-matroid and let $\alpha$ and $\beta$ be distinct edges of $BG(M)$ sharing a common end.
    Then $BG(M)$ has a Hamiltonian cycle using $\alpha$ and avoiding $\beta$ unless it is isomorphic to $C_4$ or $K_m$ with $m\in\{1,2,3\}$.
\end{theorem}

Since every matroid is an even delta-matroid, we deduce the following corollary that generalizes both Theorems~\ref{thm:Cummins} and~\ref{thm:HH}.

\begin{corollary}
    Let $M$ be a matroid and let $\alpha$ and $\beta$ be distinct edges of $BG(M)$ sharing a common end.
    Then $BG(M)$ has a Hamiltonian cycle using $\alpha$ and avoiding $\beta$ unless it is isomorphic to $C_4$ or $K_m$ with $m\in\{1,2,3\}$. \qed
\end{corollary}

In Theorem~\ref{thm:Ham}, it is necessary that $\alpha$ and $\beta$ share a common end.
For instance, let $G$ be the graph obtained from $K_3$ by adding one parallel edge, and we denote the edges of $G$ by $a,b,c,d$ such that $b$ and $d$ are parallel.
Let $M$ be the cycle matroid of~$G$.
Then its basis graph $BG(M)$ is a $4$-wheel as depicted in Figure~\ref{fig:W4} and $BG(M)$ has no Hamiltonian cycle using $\{a,b\}\{a,c\}$ and avoiding $\{b,c\}\{c,d\}$.

\begin{figure}
    \centering
    \begin{tikzpicture}
        \begin{scope}
            \node[shape=circle,fill=black, scale=0.40] (B1) at (0,0) {};
            \node[shape=circle,fill=black, scale=0.40] (B2) at (2.4,0) {};
            \node[shape=circle,fill=black, scale=0.40] (B3) at (2.4,2.4) {};
            \node[shape=circle,fill=black, scale=0.40] (B4) at (0,2.4) {};
            \node[shape=circle,fill=black, scale=0.40] (B5) at (1.2,1.2) {};
            \draw (B1) -- (B2);
            \draw[very thick] (B2) -- (B3);
            \draw (B3) -- (B4) -- (B1);
            \draw[very thick] (B1) -- (B5);
            \draw (B2) -- (B5);
            \draw (B3) -- (B5);
            \draw (B4) -- (B5);

            \node [below=0.08cm of B1] {$ab$};
            \node [below=0.08cm of B2] {$bc$};
            \node [above=0.08cm of B3] {$cd$};
            \node [above=0.08cm of B4] {$da$};
            \node [right=0.08cm of B5] {$ac$};
        \end{scope}
    \end{tikzpicture}
    \caption{The basis graph of a cycle matroid $(\{a,b,c,d\},\{\{a,b\}, \{b,c\}, \{c,d\}, \{d,a\}, \{a,b\}\})$}
    \label{fig:W4}
\end{figure}
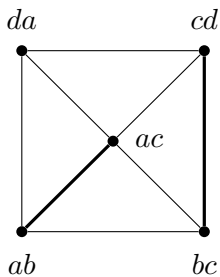

An $n$-vertex graph is \emph{vertex pancyclic} if for every vertex $v$ and every integer $3 \le k \le n$, there is a length-$k$ cycle including $v$.
A \emph{hypercube graph} $Q_d$ with nonnegative integer $d$ is a graph on $\{0,1\}^d$ such that two vertices are adjacent if and only if they differ in exactly one coordinate.
Here is our second theorem on the vertex pancyclicity of the basis graph of an even delta-matroid.
\begin{theorem}\label{thm:main}
    The basis graph of an even delta-matroid is vertex pancyclic unless it is isomorphic to a hypercube graph.
\end{theorem}

An edge $e$ of a graph $G$ is \emph{pancyclic} (respectively, \emph{almost pancyclic}) if $e$ is contained in a cycle of length $k$ for each $3\le k\le |V(G)|$ (respectively, $4 \le k \le |V(G)|$).
Components of a delta-matroid are defined in Section~\ref{sec:prelim}.
Our last theorem concerns the pancyclicity with respect to edges.

\begin{theorem}\label{thm:main2}
    Let $M$ be an even delta-matroid that has a component of size at least three.
    Then for every edge $B_1B_2$ of $BG(M)$,
    \begin{enumerate}[label=\rm(\roman*)]
        \item if $B_1\symdiff B_2$ is a component of $M$, then $B_1B_2$ is almost pancyclic but not pancyclic; and
        \item if $B_1\symdiff B_2$ is not a component of $M$, then $B_1B_2$ is pancyclic.
    \end{enumerate}
\end{theorem}

Theorem~\ref{thm:main2} implies the following theorem of Bondy and Ingleton~\cite{BI1976}.

\begin{corollary}[Bondy and Ingleton~\cite{BI1976}]
    Let $M$ be a matroid that has a component of size at least three.
    Then for every edge $B_1B_2$ of $BG(M)$,
    \begin{enumerate}[label=\rm(\roman*)]
        \item if $B_1\symdiff B_2$ is a component of $M$, then $B_1B_2$ is almost pancyclic but not pancyclic; and
        \item if $B_1\symdiff B_2$ is not a component of $M$, then $B_1B_2$ is pancyclic. \qed
    \end{enumerate}
\end{corollary}

We provide two proofs of Theorem~\ref{thm:main2} in Section~\ref{sec:ham}. 
The first one is a direct proof from various properties of even delta-matroids, and the second one uses the results of Naddef~\cite{Naddef1984} on polytopes and Chepoi~\cite{Chepoi2007} on basis graphs of even delta-matroids.

This paper is organized as follows.
In Section~\ref{sec:prelim}, we define some terminology on graphs and delta-matroids and review some properties of even delta-matroids and their basis graphs.
In Section~\ref{sec:ham}, we prove Theorems~\ref{thm:main} and~\ref{thm:main2} and we prove Theorem~\ref{thm:Ham} in Section~\ref{sec:inex}.

\section{Preliminaries}\label{sec:prelim}

For two graphs $G$ and $H$, let $G \square H$ be a graph on $V(G) \times V(H)$ such that $(x_1,y_1)$ and $(x_2,y_2)$ are adjacent if and only if either $x_1=x_2$ and $y_1y_2 \in E(H)$, or $y_1=y_2$ and $x_1x_2 \in E(G)$.
We denote the complete graph on $n$ vertices by $K_n$ and denote the cycle of length $n$ by $C_n$.

For a delta-matroid $M=(E,\cB)$, we write $E(M)$ for the \emph{ground set} $E$ of $M$, and $\cB(M)$ for the set $\cB$ of bases of~$M$.
It is easy to see that a delta-matroid $M$ is even if and only if for all bases $B,B'$ and an element $e\in B\symdiff B'$, there is $f\in (B\symdiff B')\setminus\{e\}$ such that $B\symdiff \{e,f\}\in \cB(M)$.
Wenzel~\cite{Wenzel1993} showed a stronger exchange axiom holds for even delta-matroids.

\begin{theorem}[Wenzel~\cite{Wenzel1993}]
    A pair $(E,\cB)$ of a finite set $E$ and a nonempty set $\cB$ of subsets of $E$ is an even delta-matroid if and only if 
    the following holds:
    \begin{enumerate}[label=\rm(SEA$'$)]
        \item\label{sea'} If $B, B' \in \cB$ and $e \in B \symdiff B'$, then there is $f \in (B \symdiff B')\setminus\{e\}$ such that both $B \symdiff \{e,f\}$ and $B' \symdiff \{e,f\}$ are in $\cB$.
    \end{enumerate}
\end{theorem}
One may wish that for every delta-matroid, if $B, B'$ are bases and $e \in B \symdiff B'$, then there is $f \in B \symdiff B'$ such that both $B \symdiff \{e,f\}$ and $B' \symdiff \{e,f\}$ are bases. 
However, there is a counterexample from~{\cite[Section~4.2.4]{BGW2003}}.
Let us consider a delta-matroid $(\{1,2,3\},\{\emptyset, \{1\}, \{2\}, \{3\}, \{1,2,3\}\})$ and its bases $B=\emptyset$ and $B'=\{1,2,3\}$. 
Then there is no pair of elements $e,f \in B\symdiff B'$, possibly $e=f$, such that both $B\symdiff\{e,f\}$ and $B'\symdiff\{e,f\}$ are bases.
We review some ways to construct delta-matroids.
Let $M=(E,\cB)$ be a delta-matroid. 
For a subset $X$ of $E$, let $\cB \symdiff X := \{B \symdiff X : B \in \cB\}$ and then $M \symdiff X := (E, \cB \symdiff X)$ is a delta-matroid.
We call this operation \emph{twisting on $X$}.
Evidently, if $M$ is even, then so is $M\symdiff X$.
We call $M^* := M\symdiff E$ the \emph{dual} of $M$.
An element $e\in E$ is a \emph{loop} if it is not in any base of $M$, and $e$ is a \emph{coloop} if it is in all bases of $M$.
Note that $e$ is a loop in $M$ if and only if it is a coloop in $M\symdiff \{e\}$.
For $e\in E$, the \emph{deletion of $e$ from $M$} is a pair $M \setminus e := (E\setminus\{e\}, \cB\setminus e)$ where
\[
    \cB\setminus e := 
    \begin{cases}
        \{B : e\notin B \in \cB\} & \text{if $e$ is not a coloop of $M$}, \\
        \{B\setminus\{e\}: B \in \cB\} & \text{otherwise}.
    \end{cases}
\]
It is easy to see that $M\setminus e$ is a delta-matroid, and if $M$ is even, then $M\setminus e$ is also even.
The \emph{contraction of $e$ from $M$} is a delta-matroid $M/e := (M\symdiff\{e\})\setminus e$.

If an element $e$ of an even delta-matroid $M$ is neither a loop nor a coloop, then $\cB(M)$ is partitioned into two sets, one of which induces $BG(M\setminus e)$ and the other induces a subgraph isomorphic to $BG(M/e)$.
\begin{lemma}\label{lem:del and cont}
    If $M$ is an even delta-matroid and $e\in E(M)$ is neither a loop nor a coloop, then
    \begin{enumerate}[label=\rm(\roman*)]
        \item $BG(M\setminus e)$ is equal to the subgraph of $BG(M)$ induced by $\{B : e\notin B \in \cB(M)\}$, and 
        \item $BG(M/e)$ is isomorphic to the subgraph of $BG(M)$ induced by $\{B\setminus\{e\} : e\in B \in \cB(M)\}$. \qed
    \end{enumerate}
\end{lemma}

For two delta-matroids $M_1 = (E_1,\cB_1)$ and $M_1 = (E_2,\cB_2)$ on disjoint ground sets, the \emph{direct sum} $M_1 \oplus M_2$ of $M_1$ and $M_2$ is a delta-matroid $(E_1 \cup E_2, \cB_1 \oplus \cB_2)$ where $\cB_1 \oplus \cB_2 := \{X\cup Y : X\in \cB_1 \text{ and } Y \in \cB_2\}$.
A \emph{separator} of a delta-matroid $M$ is a subset $X$ of $E(M)$ such that $M$ is the direct sum of two delta-matroids whose ground sets are $X$ and $E(M)\setminus X$.
Clearly, $\emptyset$ and $E(M)$ are always separators of $M$.
A \emph{component} of $M$ is a minimal nonempty separator.
A delta-matroid $M$ is \emph{connected} if $E(M)$ is a component or $E(M)$ is empty.
It is easy to see that a delta-matroid is even if and only if all of its components induce even delta-matroids.

Tutte~\cite{Tutte1966b} proved that for every connected matroid $M$ and every element $x\in E(M)$, $M\setminus x$ or $M/x$ is connected.
Bouchet~\cite{Bouchet2001} generalized this result to tight multimatroids. A special case of this result implies the following for even delta-matroids.

\begin{theorem}[Bouchet~\cite{Bouchet2001}]\label{thm: connectivity reduction}
    Let $M$ be an even delta-matroid.
    If $M$ is connected, 
    then for every $x\in E(M)$, $M\setminus x$ or $M/x$ is connected.
\end{theorem}

Later, we will use the following simple observation.

\begin{lemma}\label{lem:disconn}
    If a delta-matroid $M$ has a unique base, then either $M$ is disconnected or $|E(M)|\le 1$.
\end{lemma}

The components of a delta-matroid form a partition of the ground set.
For an even delta-matroid~$M$, we say $x\sim_M y$ if $x=y$ or there are elements $x_0 = x, x_1, x_2 \ldots, x_k = y$ in $E$ such that 
$\{x_{i-1},x_i\} = B_i \symdiff B_i'$ for every integer $1\le i \le k$ and some bases $B_i$ and $B_i'$ of $M$.
It is easy to see that $\sim_M$ is an equivalence relation on $E(M)$.
Then the equivalence classes of $\sim_M$ are precisely the components of $M$ as shown in the following lemma.

\begin{lemma}\label{lem:equiv}
    Let $M$ be an even delta-matroid.
    For $x,y\in E(M)$, the following are equivalent:
    \begin{enumerate}[label=\rm(\roman*)]
        \item\label{item:equiv1} $x$ and $y$ belong to the same component of $M$.
        \item\label{item:equiv2} $x\sim_M y$.
        \item\label{item:equiv3} $x=y$ or there are bases $B_1,B_2$ of $M$ such that $B_1\symdiff B_2 = \{x,y\}$.
    \end{enumerate}
\end{lemma}

\begin{proof}%
    Due to Wenzel~{\cite[Proposition~4.3]{Wenzel1996}}, \ref{item:equiv2} implies~\ref{item:equiv1}, and~\ref{item:equiv3} implies~\ref{item:equiv2} obviously.
    Thus, it suffices to show that~\ref{item:equiv1} implies~\ref{item:equiv3}.

    Suppose that~\ref{item:equiv1} holds.
    We may assume that $|E(M)| \ge 2$.
    It suffices to show that for a connected even delta-matroid $M$ and distinct elements $x,y\in E(M)$, there are bases $B_1$ and $B_2$ of $M$ such that $B_1\symdiff B_2 = \{x,y\}$.
    We proceed by induction on $|E(M)|$.
    If $|E(M)| = 2$, then it is easy to check that either $\cB(M) = \{\emptyset,\{x,y\}\}$ or $\{\{x\},\{y\}\}$ and hence the claim holds.
    Thus, we may assume that $|E(M)| \ge 3$.
    Let $z\in E(M) \setminus \{x,y\}$.
    By twisting on $\{z\}$ if necessary, we may assume that $z$ is not a coloop.
    If $M\setminus z$ is connected, then by the induction hypothesis, there are bases $B_1$ and $B_2$ of $M\setminus z$ such that $B_1\symdiff B_2 = \{x,y\}$.
    Since $B_1$ and $B_2$ are also bases of $M$, the claim holds.
    Therefore, we may assume that $M\setminus z$ is not connected.
    Then $M/z$ is connected by Theorem~\ref{thm: connectivity reduction}.
    Note that $z$ is not a loop since otherwise $M/z = M \setminus z$.
    By the induction hypothesis, there are bases $B_1'$ and $B_2'$ of $M/z$ such that $B_1'\symdiff B_2' = \{x,y\}$.
    Since $B_1 := B_1'\cup\{z\}$ and $B_2 := B_2'\cup\{z\}$ are bases of $M$, the claim holds.
\end{proof}

We remark that Wenzel~\cite{Wenzel1996} showed that~\ref{item:equiv1} and~\ref{item:equiv2} in the above lemma are equivalent. 
As a corollary, we deduce the following.

\begin{lemma}\label{lem:sep}
    For an even delta-matroid $M$, a subset $X\subseteq E(M)$ is a separator if and only if $|X\cap B|$ has the same parity for every base $B$. 
\end{lemma}
\begin{proof}
    The forward direction is obvious from the definition.
    Now let us assume that $X$ is not a separator.
    Then $M$ has a component $C$ such that both $C\cap X$ and $C\setminus X$ are nonempty.
    Then by Lemma~\ref{lem:equiv}, there are bases $B$ and $B'$ such that $B\symdiff B' = \{x,y\}$ for some $x\in C\cap X$ and $y\in C\setminus X$.
    Then $|X\cap B| \not\equiv |X\cap B'| \pmod{2}$.
\end{proof}

Bouchet~\cite{Bouchet1987} defined circuits of symmetric matroids that are equivalent to delta-matroids.
Here we review symmetric matroids under the name of lifts of delta-matroids as in~\cite{Bouchet2001}.
Let $M=(E,\cB)$ be a delta-matroid and let $E^*:=\{e^*:e\in E\}$ be a disjoint copy of $E$.
Let $(\text{-})^*$ be an involution on $E\cup E^*$ that maps $x\in E$ to $x^*\in E^*$.
For a subset $X \subseteq E\cup E^*$, we write $X^* := \{x^*:x\in X\}$.
The \emph{lift} of $M$ is a pair $\lift(M):=(E\cup E^*, \cB')$ where $\cB' := \{B\cup(E\setminus B)^* : B\in \cB\}$.
We call each member in $\cB'$ a \emph{base} of $\lift(M)$.
A \emph{circuit} of $\lift(M)$ is a minimal subset $C \subseteq E\cup E^*$ such that $C\cap C^* = \emptyset$ and $C$ is not contained in any base of $\lift(M)$.

Duchamp showed that the following result in his thesis~\cite{Duchamp1991}; see~\cite{Duchamp1995}.
Its proof can be also found in~\cite{BMP2003,Bouchet1997}.

\begin{lemma}[Duchamp~\cite{Duchamp1995}]\label{lem:ortho}
    For every pair of circuits $C$ and $D$ of the lift of a delta-matroid, we have $|C \cap D^*| \ne 1$. 
\end{lemma}

We remark that Duchamp~\cite{Duchamp1991,Duchamp1995} presented the cryptomorphic definition of the lift of a delta-matroid in terms of circuits, and its proof can be also found in~\cite{BMP2003,Bouchet1997}.
Furthermore, Booth, Moreira, and Pinto~\cite{BMP2003} gave a similar cryptomorphism for the lift of an even delta-matroid, which also follows from~{\cite[Proposition~5.4]{Bouchet1997}} and~{\cite[Theorem~4.2]{Bouchet2001}}.

\section{Pancyclicity}\label{sec:ham}

The next lemma is straightforward from the definitions.

\begin{lemma}\label{lem:components}
    Let $M$, $M_1$, $\ldots$, $M_k$ be even delta-matroids such that $M = M_1 \oplus \cdots \oplus M_k$.
    Then $BG(M)$ is isomorphic to $BG(M_1) \square \cdots \square BG(M_k)$.
    \qed
\end{lemma}

The basis graph of $(\{1,2\},\{\emptyset,\{1,2\}\})$ is isomorphic to $K_2$, and thus we deduce the following. 

\begin{lemma}\label{lem:cube}
    Let $M$ be an even delta-matroid.
    If every component of $M$ has size at most two, then $BG(M)$ is isomorphic to the hypercube graph $Q_d$ where $d$ is the number of $2$-element components of~$M$.
    \qed
\end{lemma}

We present a strengthening of Theorem~\ref{thm:main}.
A graph is \emph{strongly vertex pancyclic} if every vertex is incident with a pancyclic edge.

\begin{theorem}\label{thm:main-strong}
    Let $M$ be an even delta-matroid.
    Then the following hold:
    \begin{enumerate}[label=\rm(\roman*)]
        \item\label{main1} If $M$ has no component of size larger than two, then $BG(M)$ is isomorphic to $Q_d$, where $d$ is the number of components of size two in $M$.
        \item\label{main2} If $M$ has a component of size larger than two, then $BG(M)$ is strongly vertex pancyclic.
    \end{enumerate}
\end{theorem}

We prove Theorem~\ref{thm:main-strong} using Theorem~\ref{thm:main2}.

\begin{proof}[Proof assuming Theorem~\ref{thm:main2}]
    Lemma~\ref{lem:cube} implies~\ref{main1} and hence we now assume that $M$ has a component $X$ of size at least three.
    Let $B$ be a base of $M$. By Lemma~\ref{lem:disconn}, $M$ has a base $B' \ne B$ such that $B\symdiff B' \subseteq X$.
    By~\ref{sea}, we may assume that $|B \symdiff B'| \le 2$.
    Since $M$ is even, $|B \symdiff B'| = 2$.
    Then by Theorem~\ref{thm:main2}, an edge $BB'$ of $BG(M)$ is pancyclic, so we conclude~\ref{main2}.
\end{proof}

Now we provide two proofs of Theorem~\ref{thm:main2}, the first one is a direct proof from properties of even delta-matroids, and the second one uses the results of Naddef~\cite{Naddef1984} and Chepoi~\cite{Chepoi2007}.

\subsection{First proof of Theorem~\ref{thm:main2}}

To prove Theorem~\ref{thm:main2}, we review the following lemma, which is useful for dealing with some basic cases.

\begin{lemma}[Bondy and Ingleton~\cite{BI1976}]\label{lem: square product by complete}
    Let $G$ be a graph and $m \ge 3$ be an integer. %
    The following hold:
    \begin{enumerate}[label=\rm(\roman*)]
        \item 
        If $G$ is strongly vertex pancyclic, then every edge of the form $(u,x)(v,x)$ in $K_2 \square G$ is almost pancyclic but not pancyclic.
        \item 
        If $G$ has a Hamiltonian path starting from a vertex $x$, then every edge of the form $(u,x)(v,x)$ in $K_m \square G$ is pancyclic.
    \end{enumerate}
\end{lemma}

We say an edge $e$ of an $n$-vertex graph $G$ is \emph{even pancyclic} if $G$ has a length-$k$ cycle containing $e$ for each 
$k \in \{m\in\bZ : 3\le m\le n, \text{ ($m$ is even or $m=n$)}\}$.
One can guarantee that an edge of a graph is almost pancyclic by examining the edge pancyclicity of two subgraphs as follows.
This idea is written in the proof of Theorem~2 of Bondy and Ingleton~\cite{BI1976}.

\begin{lemma}[Bondy and Ingleton~{\cite[Proof of Theorem~2]{BI1976}}]\label{lem: larger pancyclic}
    Let $G$ be a graph and let $G_1$ and $G_2$ be its subgraphs such that $V(G)$ is partitioned into $V(G_1)$ and $V(G_2)$, and $|V(G_1)|\ge 3$.
    Let $x_iy_i \in E(G_i)$ for $i=1,2$.
    If $x_1y_1$ is pancyclic in~$G_1$, $x_2y_2$ is even pancyclic in~$G_2$, and $x_1x_2,y_1y_2\in E(G)$, then $x_1x_2$ is almost pancyclic in $G$.
\end{lemma}

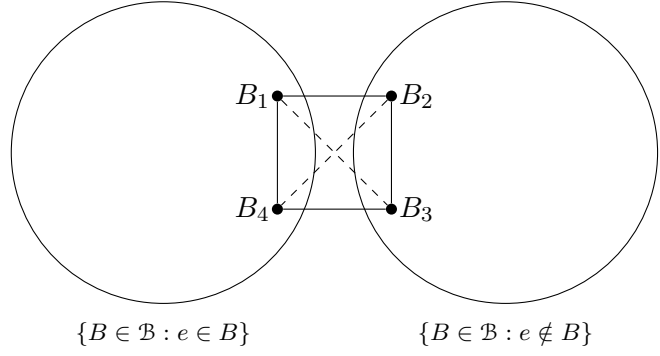
\begin{figure}
    \centering
    \begin{tikzpicture}
        \begin{scope}
            \draw (0,0) circle (2cm);
            \draw (4.5,0) circle (2cm);
            \node[shape=circle,fill=black, scale=0.40] (B1) at (1.5,0.75) {};
            \node[shape=circle,fill=black, scale=0.40] (B2) at (3,0.75) {};
            \node[shape=circle,fill=black, scale=0.40] (B3) at (3,-0.75) {};
            \node[shape=circle,fill=black, scale=0.40] (B4) at (1.5,-0.75) {};
            \draw (B1) -- (B2) -- (B3) -- (B4) -- (B1);
            \draw[dashed] (B1) -- (B3);
            \draw[dashed] (B2) -- (B4);

            \node at (1.17,0.75) {$B_1$};
            \node at (1.17,-0.75) {$B_4$};
            \node at (3.33,0.75) {$B_2$};
            \node at (3.33,-0.75) {$B_3$};

            \node at (0,-2.4) {\footnotesize{$\{B \in \cB : e\in B\}$}};
            \node at (4.5,-2.4) {\footnotesize{$\{B \in \cB : e\notin B\}$}};
        \end{scope}
    \end{tikzpicture}
    \caption{An illustration of $BG(M)$ in the proof of Theorem~\ref{thm:main2}}
    \label{fig: reduction}
\end{figure}

\begin{lemma}
    If an edge of the basis graph of an even delta-matroid is almost pancyclic, then it is even pancyclic.
\end{lemma}
\begin{proof}
    Let $M$ be an even delta-matroid.
    We may assume that $M$ has exactly three bases, since every almost pancyclic edge in a graph with at least four vertices or at most two vertices is even pancyclic.
    Suppose that $M$ has bases $B_1$ and $B_2$ such that $|B_1\symdiff B_2| \ge 4$.
    Then by~\ref{sea'}, there exist distinct $x,y \in B_1\symdiff B_2$ such that 
    $B_1\symdiff\{x,y\}$ and $B_2\symdiff\{x,y\}$ are distinct bases of $M$.
    Hence $M$ has at least four bases, a contradiction.
    Therefore, 
    $BG(M) \cong K_3$.
    Thus, every edge of $BG(M)$ is pancyclic.
\end{proof}

\begin{lemma}\label{lem:conn3}
    Let $M$ be a connected even delta-matroid on $\{1,2,3\}$.
    Then for some $X\subseteq \{1,2,3\}$,
    the set of bases of $M\symdiff X$ is equal to either $\{\{1\},\{2\},\{3\}\}$ or $\{\{1\},\{2\},\{3\},\{1,2,3\}\}$.
\end{lemma}
\begin{proof}
    By twisting, we may assume that a set $\{1\}$ is a base of $M$.
    Since $M$ is connected, $1$ is not a coloop of $M$ and so $M$ has a base not containing $1$.
    Because $M$ is even, such a base is $\{x\}$ for some $x\in \{2,3\}$.
    Let $y :=5-x$.
    As $y$ is not a loop, $M$ has a base containing $y$.
    Thus, $\{y\}$ or $\{1,x,y\}=\{1,2,3\}$ is a base of $M$ because $M$ is even.
    Then $\cB(M)$ is $\{\{1\}, \{2\}, \{3\}\}$, $\{\{1\},\{2\},\{3\},\{1,2,3\}\}$, or $\{\{1\},\{x\},\{1,2,3\}\}$.
    In the last case, $\cB(M\symdiff\{1,x\}) = \{\{1\},\{2\},\{3\}\}$.
\end{proof}

Now, we are ready to prove Theorem~\ref{thm:main2}.

\begin{proof}[Proof of Theorem~\ref{thm:main2}]
    We proceed by induction on $|E(M)|$.
    Let $e$ and $f$ be the elements of $B_1 \symdiff B_2$.
    If $|E(M)| = 3$, then $M$ is connected and by Lemma~\ref{lem:conn3}, $BG(M) \cong K_3$ or $K_4$ since twisting preserves the basis graph up to isomorphisms.
    Therefore, we may assume that $|E(M)| \ge 4$.

    By twisting, we may assume that $B_1 = B_2 \cup \{e,f\}$. %
    By Lemma~\ref{lem:equiv}, there exists a component $X$ of $M$ containing $\{e,f\}$, and we denote by $M = N\oplus M'$ such that $E(N) = X$.
    Note that by Lemma~\ref{lem:sep}, $X = \{e,f\}$ if and only if there is no base including exactly one of $e$ and $f$.

    (i) We first assume that $X = \{e,f\}$.
    Then $BG(N) \cong K_2$ and so
    by Lemma~\ref{lem:components}, $BG(M)$ is isomorphic to $K_2 \square BG(M')$.
    In addition, $M'$ has a component $Y$ of size at least three.
    Hence for every base $B$ of $M'$, there is a base $B'$ of $M'$ such that $B\symdiff B' \subseteq Y$.
    By the induction hypothesis, $BB'$ is pancyclic in $BG(M')$, implying that $BG(M')$ is strongly vertex pancyclic.
    Therefore, by Lemma~\ref{lem: square product by complete}(i), $B_1B_2$ is almost pancyclic but not pancyclic.

    (ii)
    Now we assume that $\{e,f\} \subsetneq X$. 
    By Lemma~\ref{lem:sep} and symmetry, we may assume that $M$ has a base $A$ such that $e\notin A$ and $f\in A$.
    By the exchange axiom~\ref{sea} applied to 
    $e \in B_1\symdiff A$, there is $g\in (B_1 \symdiff A)\setminus\{e\}$ such that $B_1 \symdiff \{e,g\}$ is a base of $M$.
    Note that $e,f,g$ are distinct because $f\notin B_1\symdiff A$.
    By twisting $\{g\}$ if necessary, we may assume that $g \notin B_1$.

    Suppose that $|X|=3$.
    By Lemma~\ref{lem:conn3}, $BG(N) \cong K_m$ for some $m\in \{3,4\}$ and thus $BG(M) \cong K_m \square BG(M')$. %
    If $M'$ has a component of size at least three, then by the induction hypothesis, $BG(M')$ has a Hamiltonian path starting from $B_2 \cap E(M')$.
    If every component of $M'$ has size at most two, then by Lemma~\ref{lem:cube}, $BG(M')$ has a Hamiltonian path starting from $B_2 \cap E(M')$.
    Hence $B_1B_2$ is pancyclic by Lemma~\ref{lem: square product by complete}(ii).
    Therefore, we may assume that $|X| \ge 4$.

    \begin{claim}
    There are bases $B_3$ and $B_4$ of $M$ such that
    \begin{enumerate}[label=(a\arabic*)]
        \item $e\in B_4 \ne B_1$ and $e \notin B_3 \ne B_2$;
        \item $B_{i}B_{i+1}$ is an edge in $BG(M)$ for each $i=1,2,3,4$, where $B_5 := B_1$; and
        \item $B_1B_3$ or $B_2B_4$ is an edge in $BG(M)$; see Figure~\ref{fig: reduction}.
    \end{enumerate}
    \end{claim}
    \begin{subproof}
    First, suppose that there is no base $A'$ of $M$ such that $e \in A'$ and $f\notin A'$.
    So there is no base of $\lift(M)$ which contains $\{e,f^*\}$.
    Because $e\in B_1 \cup (E(M)\setminus B_1)^*$ and $f^* \in B_2 \cup (E(M)\setminus B_2)^*$,
    neither $\{e\}$ nor $\{f^*\}$ is a circuit of $\lift(M)$.
    Thus, $\{e,f^*\}$ is a circuit of $\lift(M)$.
    By Lemma~\ref{lem:ortho}, $\{e,f,g\}$ is not a circuit of $\lift(M)$ and hence $M$ has a base $A''$ such that $e,f,g\in A''$.
    Recall that
    \[
        B_1 \cap \{e,f,g\} = \{e,f\}
        \text{ and }
        B_2 \cap \{e,f,g\} = \emptyset.
    \]
    Applying~\ref{sea} to 
    $g \in B_1 \symdiff A''$, there is $h \in (B_1 \symdiff A'') \setminus \{g\}$ such that $B_1 \symdiff \{g,h\}$ is a base.
    Because $e,f\notin B_1\symdiff A''$, four elements $e,f,g,h$ are distinct.
    Since $B_1\symdiff\{e,g\} = B_2 \symdiff\{f,g\}$, we deduce that $B_1B_2(B_1\symdiff\{e,g\})(B_1\symdiff\{g,h\})$ is a length-$4$ cycle in $BG(M)$ and $B_1(B_1\symdiff\{e,g\})$ is an edge of $BG(M)$.
    Hence the claim holds by taking $B_3 = B_1 \symdiff \{e,g\}$ and $B_4 = B_1 \symdiff \{g,h\}$.
    
    Therefore, we may assume that $M$ has a base $A'$ such that $e\in 
    A'$ and $f\notin A'$.
    By~\ref{sea} applied to 
    $f \in B_1 \symdiff A'$, there is $h \in (B_1 \symdiff A')\setminus\{f\}$ such that $B_1 \symdiff \{f,h\}$ is a base of~$M$.
    Then $h \neq e$ because $e\notin B_1 \symdiff A'$.
    If $g=h$, then by taking $B_3 = B_1 \symdiff \{e,g\} = B_2 \symdiff\{f,g\}$ and $B_4 = B_1 \symdiff \{f,g\} = B_3\symdiff\{e,f\}$ we achieve the desired properties.
    Hence we may assume that $g\ne h$.
    By~\ref{sea'} applied to $B_1 \symdiff \{e,g\}$ and $B_1 \symdiff \{f,h\}$, there is $x\in \{f,g,h\}$ such that both $B_1\symdiff\{e,g\}\symdiff\{e,x\}$ and $B_1\symdiff\{f,h\}\symdiff\{e,x\}$ are bases.
    
    If $x=g$, then we can take $B_3 = B_1\symdiff\{f,h\}\symdiff\{e,g\} = B_2 \symdiff \{g,h\}$ and $B_4 = B_1 \symdiff \{f,h\}$ to conclude the claim.
    If $x=h$, then we finish the proof by setting $B_3 = B_1 \symdiff\{e,g\} = B_2 \symdiff \{f,g\}$ and $B_4 = B_1 \symdiff\{g,h\}$.
    Finally, if $x=f$, then we can take $B_3 = B_1\symdiff\{f,h\}\symdiff\{e,f\} = B_1 \symdiff\{e,h\} = B_2 \symdiff\{f,h\}$ and $B_4 = B_1 \symdiff\{f,h\}$.
    \end{subproof}

    By (a2) and (a3), $B_1B_2$ is contained in a length-$3$ cycle of $BG(M)$.
    If $M/e$ has a component of size larger than two, then by the induction hypothesis, the edge $(B_1\setminus\{e\})(B_4\setminus\{e\})$ is almost pancyclic in $BG(M/e)$ and so it is even pancyclic.
    If $M/e$ has no component of size larger than two, then by Lemma~\ref{lem:cube}, $(B_1\setminus\{e\})(B_4\setminus\{e\})$ is even pancyclic in $BG(M/e)$.
    Similarly, the edge $B_2B_3$ is even pancyclic in $BG(M\setminus e)$.
    If $N/e$ is connected, then $E(N/e) = X\setminus\{e\}$ is a component of size at least three in $M/e$ and thus by the induction hypothesis, $(B_1\setminus\{e\})(B_4\setminus\{e\})$ is pancyclic in $BG(M/e)$.
    If $N\setminus e$ is connected, then similarly $B_2B_3$ is pancyclic in $BG(M\setminus e)$.
    By Theorem~\ref{thm: connectivity reduction}, $N/e$ or $N\setminus e$ is connected.
    Therefore, $B_1B_2$ is pancyclic in $BG(M)$ by Lemmas~\ref{lem:del and cont} and~\ref{lem: larger pancyclic}.
\end{proof}

\subsection{Second proof of Theorem~\ref{thm:main2}}\label{sec: second pf}

First, let us review the content of Naddef~\cite{Naddef1984}.
Let $E$ be a finite set and let $X \subseteq \{0,1\}^E$.
For $x,y\in \{0,1\}^E$, let $D(x,y) := \{i\in E: x_i \ne y_i\}$.
For $S\subseteq E$ and $x\in X$, let $x[S]:= (x_i:i\in S) \in \{0,1\}^S$ and let $X[S]:= \{y[S]:y\in X\}$.
A \emph{separator} of $X$ is a subset $S\subseteq E$ such that 
for every $x\in\{0,1\}^E$, $x\in X$ if and only if $x[S]\in X[S]$ and $x[E\setminus S]\in X[E\setminus S]$. 
A \emph{component} of $X$ is a minimal nonempty separator.
We say a component $C$ of $X$ is \emph{$k$-valued} if $|X[C]|=k$.

Observe that the convex hull of $X$ is a polytope whose vertices are $X$.
Let $G(X)$ be the graph on $X$ 
such that two vertices are adjacent if and only if they are joined by an edge in the convex hull of $X$. 
Note that if $D(u,v)$ is a minimal member of $\{D(x,y): x,y\in X,\; x\ne y\}$, then $uv$ is an edge in $G(X)$. %
A subset $X\subseteq \{0,1\}^E$ is \emph{nice} if for all $u,v\in X$, $uv\in E(G(X))$ if and only if $D(u,v)$ is a minimal member of $\{D(x,y): x,y\in X,\; x\ne y\}$.

Naddef~\cite{Naddef1984} showed the following result on $G(X)$.

\begin{theorem}[Naddef~{\cite[Proposition~2.5, Remark~3.7, and Theorem~3.10]{Naddef1984}}]\label{thm:Naddef}
    Let $X\subseteq \{0,1\}^E$ be a nice set and let $uv$ be an edge of $G(X)$.
    \begin{enumerate}[label=\rm(\roman*)]
        \item If $G(X)$ is bipartite, then it is a hypercube graph.
        \item If $G(X)$ is not bipartite and $D(u,v)$ is a $2$-valued component, then $uv$ is almost pancyclic but not pancyclic.
        \item If $G(X)$ is not bipartite and $D(u,v)$ is not a $2$-valued component, then $uv$ is pancyclic.
    \end{enumerate}
\end{theorem}

A graph is \emph{Hamiltonian-connected} if it contains a Hamiltonian path from $u$ to $v$ for every pair of distinct vertices $u$ and $v$.
Naddef and Pulleyblank~\cite{NP1984} proved the following.

\begin{theorem}[Naddef and Pulleyblank~\cite{NP1984}]
    \label{thm: NP}
    Let $X \subseteq \{0,1\}^E$.
    If $G(X)$ is not bipartite, then $G(X)$ is Hamiltonian-connected.
\end{theorem}

Now we construct a subset of $\{0,1\}^E$ from an even delta-matroid and deduce corollaries of the above theorems.
For $S\subseteq E$, the \emph{incidence vector} is the vector $x_S\in\{0,1\}^E$ such that $x_S(a)=1$ if $a\in S$ and $x_S(a)=0$ otherwise.
Note that $D(x_S, x_{S'}) = S\symdiff S'$ for $S,S'\subseteq E$.
For a delta-matroid $M$, let $X_M$ be the set of incidence vectors of bases of $M$.

The following lemma is to show that $X_M$ satisfies the requirement of Theorem~\ref{thm:Naddef}.

\begin{lemma}\label{lem:nice}
    For an even delta-matroid $M$, the set $X_M$ is nice.
\end{lemma}

To prove this lemma, we will use the following theorem about basis graphs of even delta-matroids.

\begin{theorem}[Chepoi~\cite{Chepoi2007}]\label{thm:dGGMS}
    For an even delta-matroid $M$, the basis graph $BG(M)$ is identical to $G(X_M)$.
\end{theorem}

\begin{proof}[Proof of Lemma~\ref{lem:nice}]
    By Theorem~\ref{thm:dGGMS}, for every edge $x_Bx_{B'}$ of $G(X_M)$, $D(x_B,x_{B'}) = B\symdiff B'$ is a $2$-element subset and so it is minimal among all $D(x_{A},x_{A'})$ with distinct bases $A$, $A'$ of $M$.

    Conversely, suppose that $B$ and $B'$ of $M$ are distinct bases such that $D(x_{B},x_{B'})$ is minimal.
    Let $x\in B\symdiff B'$.
    By \ref{sea}, there is $y\in B\symdiff B'$ such that $B\symdiff \{x,y\}$ is a base of $M$.
    Since $M$ is even, $x\neq y$.
    Observe that $D(x_{B\symdiff \{x,y\}},x_{B'})=D(x_B,x_{B'})\setminus\{x,y\}$ and therefore by the minimality, $B'=B\symdiff \{x,y\}$. This implies that $\lvert{B\symdiff B'}\rvert=2$.
\end{proof}

First, let us see a consequence of Theorem~\ref{thm:Naddef}, which will give us another proof of Theorem~\ref{thm:main2}.
For an even delta-matroid $M$, the separators of $M$ are precisely the separators of $X_M$, and for $r=1,2$, a component $C$ of $X_M$ is $r$-valued if and only if $|C|=r$.
Therefore, we deduce Theorem~\ref{thm:main2} by Theorem~\ref{thm:Naddef} and Lemma~\ref{lem:nice} along with the following lemma, immediately implied by Theorem~\ref{thm:main-strong}\ref{main2}.
We provide a direct proof of the lemma.

\begin{lemma}\label{lem: nonbip}
    If an even delta-matroid $M$ has a component of size at least three, then $BG(M)$ is not bipartite.
\end{lemma}
\begin{proof}
    If $M$ has a component of size at least three, then it has a connected minor of size three and thus $BG(M)$ has a triangle by Lemmas~\ref{lem:conn3} and~\ref{lem:del and cont}.
\end{proof}

Second, let us see an implication of Theorem~\ref{thm: NP}.
If an even delta-matroid has a component of size at least three, then its basis graph is non-bipartite by Lemma~\ref{lem: nonbip}.
Therefore, we deduce the following corollary.

\begin{corollary}\label{cor:Hamconn}
    Let $M$ be an even delta-matroid.
    If $M$ has a component of size larger than two, then $BG(M)$ is Hamiltonian-connected.
    \qed
\end{corollary}

\section{Hamiltonian cycles using $\alpha$ and avoiding $\beta$}\label{sec:inex}

We will prove Theorem~\ref{thm:Ham}, which states that 
given an even delta-matroid $M$ and distinct edges $B_1B_2$ and $B_2B_3$ of $BG(M)$, the basis graph $BG(M)$ has a Hamiltonian cycle using $B_1B_2$ and avoiding $B_2B_3$ unless $BG(M)$ is isomorphic to one of $K_1$, $K_2$, $K_3$, and $C_4$.
We first show the following stronger result when $B_1B_2B_3$ forms a triangle in $BG(M)$.

\begin{proposition}\label{prop:inex2}
    Let $M$ be an even delta-matroid.
    If $B_1B_2B_3$ is a triangle in $BG(M)$, then for each $4 \le k \le |V(BG(M))|$, $BG(M)$ has a length-$k$ cycle using $B_1B_2$ and avoiding $B_2B_3$.
\end{proposition}

\begin{proof}
    We may assume that $|V(BG(M))| \ge 4$.
    As $B_1B_2B_3$ is a triangle in $BG(M)$, $B_2 = B_1\symdiff\{e,f\}$ and $B_3 = B_1\symdiff\{e,g\}$ for some $e,f,g\in E(M)$.
    By twisting on $B_1\symdiff\{e\}$, we may assume that $B_1=\{e\}$, $B_2=\{f\}$, and $B_3=\{g\}$.
    By Lemma~\ref{lem:equiv}, there is a component $X$ of $M$ containing $\{e,f,g\}$, and then we can write $M = N \oplus M'$ such that $E(N) = X$.

    Suppose that $N = N\symdiff\{e,f\}$.
    Then $M = M\symdiff\{e,f\}$.
    Hence if $M$ has a length-$k$ cycle containing $B_1B_2$ and $B_2B_3$ with $k\ge 4$, then it has another length-$k$ cycle containing $B_2B_1$ and $B_1B_3$, which does not contain $B_2B_3$.
    Therefore, $M$ has a length-$k$ cycle using $B_1B_2$ and avoiding $B_2B_3$ for each $4\le k\le |V(BG(M))|$.

    Now we assume that $N \ne N\symdiff\{e,f\}$.
    Then $|X| > 3$ by Lemma~\ref{lem:conn3}.
    Since $M \ne M\symdiff\{e,f\}$, $M$ has a minimal base $A$ such that $|A|\ge 3$ and $A$ contains $e$ or $f$.
    Then by~\ref{sea}, $|A|=3$.
    If $A$ does not contain $e$, then by applying~\ref{sea'} to $\{e\}$ and $A$, we obtain a base $A'$ such that $e\in A'$ and $|A'| = |A|$.
    Thus, we may assume that $e\in A$.

    \begin{claim}
        There are bases $A_1$ and $A_2$ of $M$ such that
    \begin{enumerate}[label=\rm(a\arabic*)]
        \item $f\notin A_1$ and $f\in A_2$, and
        \item $B_1B_2A_2A_1$ is a length-$4$ cycle in $BG(M)$.
    \end{enumerate}
    \end{claim}

    \begin{subproof}
    Suppose that $f\notin A$.
    By the exchange axiom~\ref{sea} applied to $A$ and $\{f\}$, we deduce that $A\symdiff\{x,f\}$ is a base for some $x\in A$.
    Then we can take $A_1 = A$ and $A_2 = A\symdiff\{x,f\}$.

    Therefore, we may assume that $e,f\in A$.
    Then $A=\{e,f,h\}$ for some $h\in E(M) \setminus\{e,f\}$.
    If $h=g$, then %
    we can take $A_1 = \{g\}$ and $A_2 = \{e,f,g\}$.
    
    Thus, we may assume that $h\ne g$.
    Then by~\ref{sea} applied to 
    $\{e,f,h\}$ and $B_3 = \{g\}$,
    there is $z\in \{e,f,h\}$ such that 
    $\{e,f,g,h\}\setminus \{z\}$ is a base of $M$.
    If $z=e$ or $h$, then we can take $A_1=\{g\}$ and $A_2=\{e,f,g,h\}\symdiff\{z\}$.
    If $z=f$, then we can take $A_1=\{e,g,h\}$ and $A_2=\{e,f,h\}$.
    \end{subproof}

    If $M\setminus f$ has a component of size at least three, then by Theorem~\ref{thm:main2}, the edge $B_1A_1$ is almost pancyclic in $BG(M\setminus f)$ and so even pancyclic.
    If $M\setminus f$ has no component of size larger than two, then by Lemma~\ref{lem:cube}, $B_1A_1$ is even pancyclic in $BG(M\setminus f)$.
    Similarly, $(B_2\setminus\{f\})(A_2\setminus\{f\})$ is even pancyclic in $BG(M/ f)$.
    By Theorem~\ref{thm: connectivity reduction}, $N\setminus f$ or $N/f$ is connected, and so $M\setminus f$ or $M/f$ has a component $X\setminus\{f\}$ of size at least three.
    Then by Theorem~\ref{thm:main2}, $B_1A_1$ is pancyclic in $BG(M\setminus f)$ or $(B_2\setminus\{f\})(A_2\setminus\{f\})$ is pancyclic in $BG(M/f)$.
    Let $F:=\{BB' \in E(BG(M)) : f\in B\symdiff B'\} \setminus \{B_1B_2, A_1A_2\}$.
    Then $B_2B_3 \in F$ and $BG(M)$ is isomorphic to a graph obtained from the disjoint union of $BG(M\setminus f)$ and $BG(M/ f)$ by adding two edges corresponding to $B_1B_2$ and $A_1A_2$.
    By Lemma~\ref{lem: larger pancyclic} applied to $BG(M) \setminus F$, the edge $B_1B_2$ is almost pancyclic in $BG(M)\setminus F$ and thus $BG(M)$ has a length-$k$ cycle using $B_1B_2$ and avoiding $B_2B_3$ for each $4\le k\le |V(BG(M))|$.
\end{proof}

We cannot omit the condition that $B_1B_2B_3$ is a triangle in Proposition~\ref{prop:inex2}, because the basis graph in Figure~\ref{fig:W4} does not have a length $4$-cycle that uses $\{a,b\}\{a,c\}$ and avoids $\{a,c\}\{c,d\}$.
We also note that this basis graph has a unique triangle containing an edge $\{a,b\}\{b,c\}$ and thus $k$ cannot be $3$ in Proposition~\ref{prop:inex2}.

Finally, we prove Theorem~\ref{thm:Ham}.

\begin{proof}[Proof of Theorem~\ref{thm:Ham}]
    Let $B_1,B_2,B_3$ be the vertices of $BG(M)$ such that
    $\alpha = B_1B_2$ and $\beta = B_2B_3$.
    By twisting, we may assume that $B_2 = \emptyset$.
    Then $B_1 = \{e,f\}$ and $B_3 = \{g,h\}$ for some elements $e,f,g,h \in E$.
    By Proposition~\ref{prop:inex2}, we can assume that $\{e,f\}$ and $\{g,h\}$ are disjoint.

    Let $G_1$ be the subgraph of $BG(M)$ induced by the set of all bases of $M$ containing $e$, and let $G_2$ be the subgraph of $BG(M)$ induced by the set of all bases of $M$ not containing $e$.
    Then 
    $G_1 \cong BG(M/e)$ and $G_2 = BG(M\setminus e)$.

    By~\ref{sea'} applied to $B_1$ and $B_3$, at least one of the following holds:
    \begin{enumerate}[label=\rm(a\arabic*)]
        \item $\{e,g\}$ and $\{f,h\}$ are bases.
        \item $\{e,h\}$ and $\{f,g\}$ are bases.
        \item $\emptyset$ and $\{e,f,g,h\}$ are bases.
    \end{enumerate}
    By exchanging $g$ and $h$ if necessary, we may assume that (a1) or (a3) holds.
    Suppose (a1) holds.
    Then we observe that $B_1B_2B_3\{e,g\}$ is a length-$4$ cycle in $BG(M)$.
    By Theorem~\ref{thm:main2} and Lemma~\ref{lem:cube}, $G_1$ has a Hamiltonian path $P_1$ from $B_1=\{e,f\}$ to $\{e,g\}$ and $G_2$ has a Hamiltonian path $P_2$ from $B_2=\emptyset$ to $B_3=\{g,h\}$.
    Because $B_2, B_3, \{f,h\}\in  V(G_2)$, the path $P_2$ does not contain the edge $B_2B_3$.
    Then the edge set $E(P_1)\cup E(P_2) \cup \{B_1B_2, B_3\{e,g\}\}$ forms a Hamiltonian cycle of $BG(M)$, which does not contain $B_2B_3$.

    Therefore, we can assume (a3) holds.
    Denote $B_4 := \{e,f,g,h\}$.
    Since $BG(M) \not\cong C_4$, $M$ has a base other than $B_i$ with $1\le i\le 4$.
    Suppose that $M$ has a base $A$ such that $\{e,f\} \not\subseteq A$, and let $x\in \{e,f\}\setminus A$.
    By symmetry between $e$ and $f$, we may assume that $x=e$.
    By Theorem~\ref{thm:main2} and Lemma~\ref{lem:cube}, $G_1$ has a Hamiltonian path $P_1$ from $B_1$ to $B_4$ and $G_2$ has a Hamiltonian path $P_2$ from $B_2$ to $B_3$.
    As $B_2,B_3,A\in V(G_2)$,
    we have that $B_2B_3 \notin E(P_2)$ and thus $E(P_1)\cup E(P_2)\cup \{B_1B_2,B_3B_4\}$ is a Hamiltonian cycle of $M$ which avoids $B_2B_3$.
    So we may assume that every base other than $B_2$ and $B_3$ contains $\{e,f\}$.
    Let $A'$ be a base of $M$ such that $\{e,f\} \subseteq A' \notin \{B_1,B_4\}$.
    By~\ref{sea}, there is $y \in (A' \symdiff B_2) \setminus \{e\}$ such that $A'\symdiff\{e,y\}$ is a base of $M$.
    Then $A'\symdiff\{e,y\}$ does not contain $e$ and it is neither $B_2$ nor $B_3$, a contradiction.
\end{proof}

\section{Acknowledgements}
The authors would like to thank the anonymous reviewers for their valuable comments and suggestions.
Especially, Lemma~\ref{lem:equiv}\ref{item:equiv3} and Corollary~\ref{cor:Hamconn} were suggested by the reviewers.

\providecommand{\bysame}{\leavevmode\hbox to3em{\hrulefill}\thinspace}
\providecommand{\MR}{\relax\ifhmode\unskip\space\fi MR }
\providecommand{\MRhref}[2]{%
  \href{http://www.ams.org/mathscinet-getitem?mr=#1}{#2}
}
\providecommand{\href}[2]{#2}

\end{document}